\documentclass[a4paper,11pt,oneside]{amsart}
\pdfoutput=1

\usepackage[utf8]{inputenc}
\usepackage{bm}
\usepackage{mathtools,amssymb}
\usepackage{esint}
\usepackage{hyperref}
\hypersetup{
    colorlinks=true,
    linkcolor=black,
    filecolor=black,      
    urlcolor=cyan,
    citecolor=black
    }
\usepackage{tikz}
\usepackage{dsfont}
\usepackage{relsize}
\usepackage{url}
\usepackage{xcolor}
\usepackage{graphicx}
\usepackage{mathrsfs}
\usepackage[shortlabels]{enumitem}
\usepackage{lineno}
\usepackage{amsmath}
\usepackage{enumitem}
\usepackage{amsthm} 
\usepackage{verbatim}
\usepackage{dsfont}

\allowdisplaybreaks

\mathtoolsset{showonlyrefs}

\graphicspath{{images/}}

\newtheorem{theorem}{Theorem}[section]

\newtheorem{remark}[theorem]{Remark}

\title[Inverse fractional conductivity problem]{Counterexamples to uniqueness in the inverse fractional conductivity problem with partial data}

\author{Jesse Railo}
\thanks{Department of Pure Mathematics and Mathematical Statistics, University of
Cambridge (\href{mailto:jr891@cam.ac.uk}{jr891@cam.ac.uk})}
\address{Department of Pure Mathematics and Mathematical Statistics, University of
Cambridge, Cambridge CB3 0WB, UK}
\email{jr891@cam.ac.uk}
\author{Philipp Zimmermann}
\thanks{Department of Mathematics, ETH Zurich (\href{mailto:philipp.zimmermann@math.ethz.ch}{philipp.zimmermann@math.ethz.ch})}
\address{Department of Mathematics, ETH Zurich, Z\"urich, Switzerland}
\email{philipp.zimmermann@math.ethz.ch}
\date{\today}

\newcommand{\R}{{\mathbb R}}
\newcommand{\Z}{{\mathbb Z}}

\newcommand{\schwartz}{\mathscr{S}}

\newcommand{\tempered}{\mathscr{S}^{\prime}}

\newcommand{\fourier}{\mathcal{F}}
\newcommand{\ifourier}{\mathcal{F}^{-1}}
\newcommand{\vev}[1]{\left\langle#1\right\rangle}

\newcommand{\Hcirc}{\overset{\hspace{-0.08cm}\circ}{H^s}}

\newcommand{\norm}[1]{\lVert #1 \rVert}
\newcommand{\abs}[1]{\left\lvert #1 \right\rvert}
\DeclareMathOperator{\Div}{div} 
\DeclareMathOperator{\supp}{supp} 

\begin{document}

\maketitle
\begin{abstract} We construct counterexamples for the partial data inverse problem for the fractional conductivity equation in all dimensions on general bounded open sets. In particular, we show that for any bounded domain $\Omega \subset \R^n$ and any disjoint open sets $W_1,W_2 \Subset \R^n \setminus \overline{\Omega}$ there always exist two positive, bounded, smooth, conductivities $\gamma_1,\gamma_2$, $\gamma_1 \neq \gamma_2$, with equal partial exterior Dirichlet-to-Neumann maps $\Lambda_{\gamma_1}f|_{W_2} = \Lambda_{\gamma_2}f|_{W_2}$ for all $f \in C_c^\infty(W_1)$. The proof uses the characterization of equal exterior data from another work of the authors in combination with the maximum principle of fractional Laplacians. The main technical difficulty arises from the requirement that the conductivities should be strictly positive and have a special regularity property $\gamma_i^{1/2}-1 \in H^{2s,\frac{n}{2s}}(\R^n)$ for $i=1,2$. We also provide counterexamples on domains that are bounded in one direction when $n \geq 4$ or $s \in (0,n/4]$ when $n=2,3$ using a modification of the argument on bounded domains.
\end{abstract}

\section{Introduction}

The Calderón problem for the conductivity equation is the mathematical model of electrical impedance tomography (EIT) \cite{UHL-electrical-impedance-tomography}. One considers the conductivity equation 
\begin{equation}
\label{eq: cond eq}
    \begin{split}
            \Div(\gamma\cdot\nabla u)&= 0\quad\text{in}\quad\Omega,\\
            u&= f\quad\text{on}\quad \partial\Omega,
        \end{split}
\end{equation}
where $\gamma$ models the conductivity of a medium. In the inverse conductivity problem, one sets a voltage $f$ on the boundary $\partial \Omega$ and measures the current $\Lambda_\gamma f := \nu \cdot \gamma\nabla u_f|_{\partial \Omega}$ on the boundary where $u_f$ is the unique solution to the problem \eqref{eq: cond eq} with the boundary condition $f$ and the vector field $\nu$ is the outer unit normal of $\partial \Omega$. The Calderón problem is to determine $\gamma$ or its properties from the knowledge $f \mapsto \Lambda_\gamma f$.

In this article, we study the analogous problem for the fractional conductivity equation \eqref{eq: frac cond eq}. The fractional conductivity equation models a nonlocal diffusion, which is defined in the whole Euclidean space $\R^n$ rather than only locally \cite{NonlocDiffusion}. The fractional conductivity equation also shows up in weighted long jump random walk models \cite{covi2019inverse-frac-cond}. In these random walk models, the larger the conductivity at some location, the higher the probability for a particle to jump from that location.

We say that an open set $\Omega_\infty \subset\R^n$ of the form $\Omega_\infty=\R^{n-k}\times \omega$, where $n\geq k\geq 1$ and $\omega \subset \R^k$ is a bounded open set, is a \emph{cylindrical domain}. We say that an open set $\Omega \subset \R^n$ is \emph{bounded in one direction} if there exists a cylindrical domain $\Omega_\infty \subset \R^n$ and a rigid Euclidean motion $A(x) = Lx + x_0$, where $L$ is a linear isometry and $x_0 \in \R^n$, such that $\Omega \subset A\Omega_\infty$. Fractional Calderón problems and Poincaré inequalities for the fractional Laplacians in Bessel potential spaces were recently studied in such domains in \cite{RZ2022unboundedFracCald}.

We assume that $0<s<\min(1,n/2)$, $\Omega\subset \R^n$ is generally an open set which is bounded in one direction and denote by $\Omega_e\vcentcolon =\R^n\setminus\overline{\Omega}$ the exterior of $\Omega$. If $\gamma \in L^\infty(\R^n)$ is a positive conductivity, then we denote by $\Theta_{\gamma}(x,y)\vcentcolon =\gamma^{1/2}(x)\gamma^{1/2}(y)\mathbf{1}_{n\times n}$ for $x,y\in\R^n$ the conductivity matrix and by $m_{\gamma}\vcentcolon =\gamma^{1/2}-1$ the background deviation. The fractional gradient is defined for all sufficiently regular functions by the formula
    \[
        \nabla^su(x,y)=\sqrt{\frac{C_{n,s}}{2}}\frac{u(x)-u(y)}{|x-y|^{n/2+s+1}}(x-y)
    \]
and $\Div_s$ denotes its adjoint operator. In particular, $\Div_s(\nabla^su) = (-\Delta)^su$ in the weak sense for all $u \in H^s(\R^n)$.

We say that $u\in H^s(\R^n)$ is a weak solution to the fractional conductivity equation
\begin{equation}
\label{eq: frac cond eq}
    \begin{split}
            \Div_s(\Theta_{\gamma}\nabla^s u)&= F\quad\text{in}\quad\Omega,\\
            u&= f\quad\text{in}\quad\Omega_e,
        \end{split}
\end{equation}
where $f\in H^s(\R^n)$, $F\in(\tilde{H}^s(\Omega))^*$, if $u-f\in\tilde{H}^s(\Omega)$ and there holds
\begin{equation}
    B_{\gamma}(u,v)\vcentcolon = \int_{\R^{2n}}\Theta_{\gamma}\nabla^su\cdot\nabla^sv\,dxdy=F(v)
\end{equation}
for all $v\in \tilde{H}^s(\Omega)$. If $\gamma\geq \gamma_0>0$ and $m_{\gamma}\in H^{2s,\frac{n}{2s}}(\R^n)$, then the exterior Dirichlet-to-Neumann (DN) map related to the fractional conductivity equation \eqref{eq: frac cond eq} is the bounded linear operator $\Lambda_{\gamma}\colon X\to X^*$ given by 
        \[
        \begin{split}
            \langle \Lambda_{\gamma}f,g\rangle \vcentcolon =B_{\gamma}(u_f,g)
        \end{split}
        \]
        where $u_f\in H^s(\R^n)$ is the unique solution to the homogeneous fractional conductivity equation with exterior value $f$ and $X\vcentcolon = H^s(\R^n)/\Tilde{H}^s(\Omega)$ is the (abstract) trace space (see~\cite[Lemma 8.10]{RZ2022unboundedFracCald}). 

Now the inverse problem related to the fractional conductivity equation consists in showing that two sufficiently nice conductivities $\gamma_1,\gamma_2$ coincide in $\R^n$ if the corresponding DN maps satisfy $\Lambda_{\gamma_1}f|_{W_2}=\Lambda_{\gamma_2}f|_{W_2}$ for all $f\in C_c^{\infty}(W_1)$ and some nonempty open sets $W_1,W_2\subset \Omega_e$. In our recent article, we have shown the following theorem. The proof uses ideas from the works \cite{covi2019inverse-frac-cond,GSU20,RS-fractional-calderon-low-regularity-stability} in combination with the exterior determination arguments developed in \cite{RZ2022unboundedFracCald} for globally defined fractional conductivities.

\begin{theorem}[{\cite[Theorem 2.8]{RZ2022unboundedFracCald}}]
\label{thm: characterization of uniqueness}
    Let $\Omega\subset \R^n$ be an open set which is bounded in one direction and $0<s<\min(1,n/2)$. Assume that $\gamma_1,\gamma_2\in L^{\infty}(\R^n)$ with background deviations $m_1,m_2$ satisfy $\gamma_1(x),\gamma_2(x)\geq \gamma_0>0$ and $m_1,m_2\in H^{2s,\frac{n}{2s}}(\R^n)$. Moreover, assume that $m_0\vcentcolon =m_1-m_2\in H^s(\R^n)$ and $W_1,W_2\subset\Omega_e$ are nonempty open sets with
    \begin{equation}\label{eq:conductivitysupportassumptions}
         (\supp(m_1)\cup\supp(m_2))\cap (W_1\cup W_2)=\emptyset.
    \end{equation}
Then the following statements hold:
\begin{enumerate}[(i)]
    \item\label{item 1 characterization of uniqueness} If $W_1\cap W_2\neq \emptyset$, then  $\left.\Lambda_{\gamma_1}f\right|_{W_2}=\left.\Lambda_{\gamma_2}f\right|_{W_2}$ for all $f\in C_c^{\infty}(W_1)$ if and only if $\gamma_1=\gamma_2$ in $\R^n$.
    \item\label{item 2 characterization of uniqueness} If $W_1\cap W_2=\emptyset$, then $\left.\Lambda_{\gamma_1}f\right|_{W_2}=\left.\Lambda_{\gamma_2}f\right|_{W_2}$ for all $f\in C_c^{\infty}(W_1)$ if and only if $m_1-m_2$ is the unique solution of
    \begin{equation}
    \label{eq: PDE uniqueness cond eq main thm}
        \begin{split}
            (-\Delta)^sm-\frac{(-\Delta)^sm_1}{\gamma_1^{1/2}}m&=0\quad\text{in}\quad \Omega,\\
            m&=m_0\quad\text{in}\quad \Omega_e.
        \end{split}
    \end{equation}
\end{enumerate}
\end{theorem}

In particular, this theorem shows that if the measurement sets $W_1,W_2\subset \Omega_e$ have nonempty intersection, then the inverse problem is uniquely solvable (see the very recent work \cite[Theorems 1.1 and 1.3]{CRZ22} for a general uniqueness result without assuming that the conductivities are constant in the sets $W_1$ and $W_2$). This leads to the question whether one can actually construct a counterexample in the situation when the measurement sets are disjoint. The geometric setting of Theorem \ref{thm: characterization of uniqueness} \ref{item 2 characterization of uniqueness} is illustrated in Figure~ \ref{fig: Geometric setting 1} below.
 \begin{figure}[!ht]
    \centering
    \begin{tikzpicture}
    \draw [cyan, xshift=4cm]   (0,-1)--(1.5,-1)--(1.5,1.75)--(2.5,1.75)--(2.5,-0.25)--(4,-0.25)-- (4,0.75)--(6,0.75)-- (6,2) --(7,2)-- (7,0) --(10,0)--(10,-2);
    \draw [cyan, xshift=4cm]  (0,-1) --(0,-2.5)--(1.5,-2.5)--(1.5,-5.25)--(2.5,-5.25)--(2.5,-3.25)--(4,-3.25)-- (4,-4.25)--(6,-4.25)-- (6,-5.5) --(7,-5.5)-- (7,-3.5) --(10,-3.5)--(10,-2);
    \fill[cyan!5, xshift=4cm] (0,-1)--(1.5,-1)--(1.5,1.75)--(2.5,1.75)--(2.5,-0.25)--(4,-0.25)-- (4,0.75)--(6,0.75)-- (6,2) --(7,2)-- (7,0) --(10,0)--(10,-3.5)-- (7,-3.5)--(7,-5.5)-- (6,-5.5)--(6,-4.25)-- (4,-4.25)--(4,-3.25)--(2.5,-3.25)--(2.5,-5.25)--(1.5,-5.25)--(1.5,-2.5)--(0,-2.5)-- (0,-1);
    \filldraw[color=green!50, fill=green!10, xshift=2.5cm,yshift=-1cm](2,1.25) circle (0.7);
    \filldraw[color=blue!50, fill=blue!5, xshift=12cm, yshift=-2cm] (3,-2.5) ellipse (1 and 0.75);
    \node[xshift=12cm, yshift=-2cm] at (3,-2.5) {$\raisebox{-.35\baselineskip}{\Large\ensuremath{W_2}}$};
    \node[xshift=2.5cm, yshift=-1cm] at (2,1.25) {$\raisebox{-.35\baselineskip}{\Large\ensuremath{W_1}}$};
    \node[xshift=-1.4cm] at (6,-1.75) {$\raisebox{-.35\baselineskip}{\Large\ensuremath{\Omega}}$};
    \filldraw [color=orange!80, fill = orange!5, xshift=8cm, yshift=-2cm,opacity=0.8] plot [smooth cycle] coordinates {(-2,0.75)(-2,3) (-1,1) (3,1.5) (6,-3.5) (3,-1) (-3,-0.25)};
     \filldraw [color=red!60, fill = red!5, xshift=10cm, yshift=-2cm, opacity=0.8] plot [smooth cycle] coordinates {(0,3) (0,4.5) (1,4) (2.2,4) (2.2,1.8) (5.6,1.65) (4, 0.5) (3.65, -1.3)(3.1,-0.7) (1,0)};
    \node[xshift=2cm] at (7,-1.75) {$\raisebox{-.35\baselineskip}{\Large\ensuremath{\supp\,m_1}}$};
    \node[xshift=2cm] at (10.7,-1.05) {$\raisebox{-.35\baselineskip}{\Large\ensuremath{\supp\,m_2}}$};
\end{tikzpicture}
    \caption{A graphical illustration of a possible geometric setting of Theorem~\ref{thm: characterization of uniqueness}. Here $\Omega\subset\R^n$ is taken for simplicity to be a bounded domain and the measurements are performed in the disjoint nonempty open subsets $W_1,W_2\subset \Omega_e$. Moreover, the supports of the background deviations $m_1,m_2$, which are represented in orange and red, respectively, do not necessarily coincide in the exterior $\Omega_e$ and can even be disjoint as illustrated above. In this case uniqueness may be lost.}
    \label{fig: Geometric setting 1}
\end{figure}
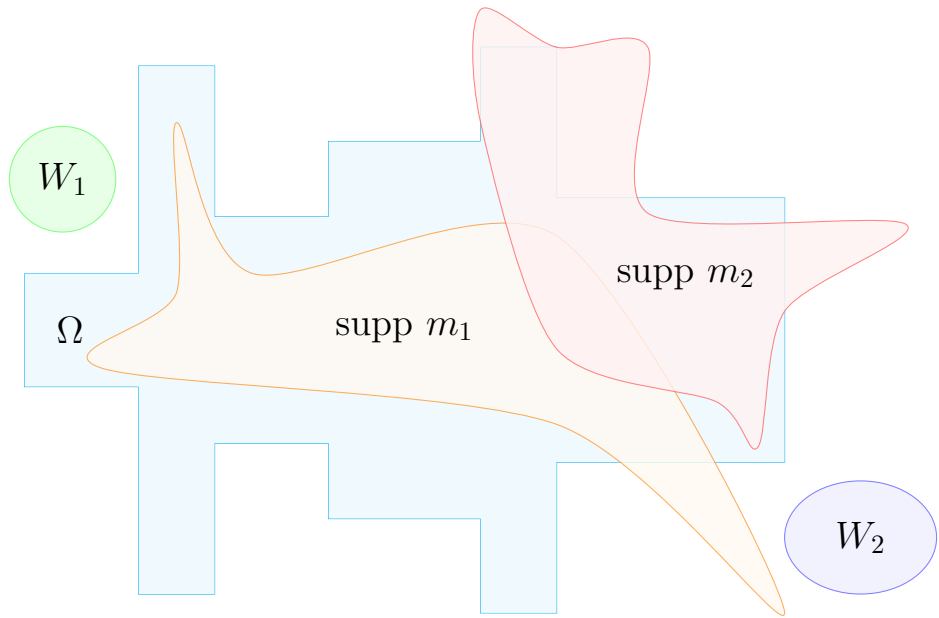

We give a positive answer to this question (the question remains open if $W_1$ or $W_2$ has limit points on the boundary $\partial \Omega$) and prove the following two theorems in this article.

\begin{theorem}\label{thm: counterexample}
    For any open bounded set $\Omega\subset \R^n$, $ 0<s<\min(1,n/2)$ and nonempty open disjoint sets $W_1,W_2\Subset\Omega_e$ there exist two different conductivities $\gamma_1,\gamma_2\in L^{\infty}(\R^n) \cap C^\infty(\R^n)$ satisfying the assumptions of Theorem~\ref{thm: characterization of uniqueness} and $\left.\Lambda_{\gamma_1}f\right|_{W_2}=\left.\Lambda_{\gamma_2}f\right|_{W_2}$ for all $f\in C_c^{\infty}(W_1)$.
\end{theorem}

\begin{theorem}\label{thm: counterexample 2}
    For any open set $\Omega\subset \R^n$ which is bounded in one direction, $0<s<\min(1,n/2)$ (with $s\leq n/4$ when $n=2,3$) and nonempty open disjoint sets $W_1,W_2\Subset\Omega_e$ there exist two different conductivities $\gamma_1,\gamma_2\in L^{\infty}(\R^n) \cap C^\infty(\R^n)$ satisfying the assumptions of Theorem~\ref{thm: characterization of uniqueness} and $\left.\Lambda_{\gamma_1}f\right|_{W_2}=\left.\Lambda_{\gamma_2}f\right|_{W_2}$ for all $f\in C_c^{\infty}(W_1)$.\footnote{This theorem was recently generalized by the authors to cover also the cases $n=2,3$ and $n/4 < s < 1$ when $m_1,m_2 \in H^{s,n/s}(\R^n)$ instead of $H^{2s,\frac{n}{2s}}(\R^n)$ \cite[Theorem 1.6]{RZ22LowReg}.}
\end{theorem}

\begin{remark} The constructions given in the proofs of Theorem~\ref{thm: counterexample} and \ref{thm: counterexample 2} actually show that there are uncountably many different conductivities which have equal partial exterior DN data. In the setting of Theorem~\ref{thm: counterexample} one can also take these conductivies such that they take the constant value one, $\gamma(x) = 1$, outside a bounded set. We also remark that the constructed conductivities are really counterexamples rather than some expected gauge associated with the equation and the inverse problem since such obstructions for the uniqueness are not present when $W_1 \cap W_2 \neq \emptyset$.
\end{remark}

\begin{remark} If $W, X \subset \R^n$ are open sets, we write $W \Subset X$ whenever $\overline{W} \subset X$ and $\overline{W}$ is compact. In Theorems~\ref{thm: counterexample} and \ref{thm: counterexample 2} one may replace the assumption $W_1,W_2 \Subset \Omega_e$ by the weaker condition that $\overline{W_1 \cup W_2} \subset \Omega_e$ and 
$$\mathrm{dist}(W_1\cup W_2,\Omega) \vcentcolon = \inf_{x \in W_1\cup W_2, y \in \Omega} \abs{x-y} > 0$$ and the conclusion would still remain true. This can be obtained by making minor changes in the proofs.
\end{remark}

In the following subsection, we discuss earlier literature on the related problems in more details. Preliminaries and basic notation of fractional order Sobolev spaces and operators are recalled in Section \ref{sec: preliminaries}. We discuss the invariance of exterior DN data in Section \ref{sec: invariance of data} as a formal preparation for the counterexamples to uniqueness, which are constructed in Section \ref{sec: counterexamples}. We point to the article \cite{RZ2022unboundedFracCald} for the details how to prove the characterization of Theorem \ref{thm: characterization of uniqueness} that plays a fundamental role behind the constructions of counterexamples given in this article.

\subsection{On the earlier literature}

We briefly introduce some of the important contributions related to the fractional Calderón problems. We also discuss on nonuniqueness in other Calderón type inverse problems. This discussion is far from complete and more references can be found from the discussed works.

\subsubsection{Fractional Calderón problems for perturbations}

The study of fractional Calderón problems started from the celebrated work of Ghosh, Salo and Uhlmann \cite{GSU20} where they showed that the knowledge of the partial exterior DN map associated with the fractional Schrödinger equation, $0 < s < 1$,
\begin{equation}
\begin{split}\label{eq:fracSchördinger}
    ((-\Delta)^s+q)u &= 0\quad \text{in}\quad \Omega, \\
    u&=f\quad \text{in}\quad \Omega_e
\end{split}
\end{equation}
determines the potential $q \in L^\infty(\Omega)$ uniquely. The solution to the inverse problem is based on the unique continuation property of the fractional Laplacian and a clever observation that it implies the Runge approximation property for the equation, which in turn can be used to solve the inverse problem. This and related problems have been studied actively ever since. It is known that one may determine the potential $q$ using just a single measurement \cite{GRSU-fractional-calderon-single-measurement,R-Singular-measurement} and the inverse problem has logarithmic stability \cite{RS-fractional-calderon-low-regularity-stability}. This stability estimate is also known to be optimal \cite{RS-exponential-instability,RS-Instability}. One may replace $q$ by a general local bilinear form $Q$, $(-\Delta)^s$ by a general nonlocal bilinear form $L$ having the unique continuation property, and solve the inverse problem on unbounded domains whenever the forward problem is well-posed \cite{RZ2022unboundedFracCald}. For example, $Q$ may be taken to be a Sobolev multiplier \cite{RS-fractional-calderon-low-regularity-stability} or given by a general lower order partial differential operator (PDO) \cite{CMRU20-higher-order-fracCald}, and $L$ can be given by a fractional power of an elliptic operator \cite{GLX-calderon-nonlocal-elliptic-operators,ghosh2021calderon}. One may also solve the related inverse problems for higher order equations, $s \in \R_+ \setminus \Z$, due to the work \cite{CMR20}. There are also examples where the inverse problem can be solved when $Q$ is given by a linear PDO having higher order than the fractional Laplacian \cite{RZ2022unboundedFracCald}. One can also replace $q$ in some cases by local nonlinear PDOs \cite{LL-fractional-semilinear-problems}. One may also consider linear perturbations $q$, which are assumed to be quasilocal operators or operators of finite propagation \cite{covi2021uniquenessQuasiLocal}. Unique continuation properties for general classes of operators and inverse problems for their perturbations are studied recently in \cite{covietal2021calderon-directionally-antilocal}.

\subsubsection{Fractional Calderón problems for variable coefficient operators}
We next discuss the Calderón problem for the fractional conductivity equation and other variable coefficient operators, rather than for nonlocal models with perturbations. The inverse problem for the conductivity equation \eqref{eq: frac cond eq} was studied first by Covi in \cite{covi2019inverse-frac-cond} where it is proved that if $\Omega$ is a bounded Lipschitz domain, $\gamma_1=\gamma_2 \equiv 1$ in $\Omega_e$ and any partial exterior DN data agrees (even disjoint $W_1$ and $W_2$ are possible), then $\gamma_1=\gamma_2$ in $\R^n$. The proof in \cite{covi2019inverse-frac-cond} is based on introducing the fractional Liouville transformation between the equations \eqref{eq: frac cond eq} and \eqref{eq:fracSchördinger}, which works analogously to the classical Calderón problem \cite{SU87-CalderonProblem-annals}. The assumption of \cite{covi2019inverse-frac-cond} is analogous to the classical Calderón problem under the asssumption that the conductivity takes the constant value one on the boundary. The assumption of Theorem \ref{thm: characterization of uniqueness} is analogous to the assumption that conductivities take the constant value one only on the sets where we measure and the counterexamples of Theorem \ref{thm: counterexample} work under such assumptions. A closely related fractional Schrödinger equation with magnetic field was studied and introduced in \cite{CO-magnetic-fractional-schrodinger}. Its higher order versions and higher order fractional gradient operators were studied recently in \cite{CMR20}. There are also other studies related to fractional magnetic operators with variable coefficients (see~e.g.~ \cite{LILI-fractional-magnetic-calderon} and references therein). Inverse problems for fractional diffusion equations with variable coefficients has been studied in \cite{LiLi-power-type-nonlin}. The Calderón problem for the fractional spectral Laplacian on closed Riemannian manifolds was recently solved in \cite{feizmohammadi2021fractional,feizmohammadiEtAl2021fractional}. The Calderón problem for the coefficients of a general fractional power of an elliptic second order operator from its exterior data was recently studied in \cite{ghosh2021calderon}.

\subsubsection{Nonuniqueness and gauge invariance in (fractional) Calderón problems}

Anisotropic Calderón problems and Calderón problems on Riemannian manifolds have always some natural gauge invariance, which keeps the DN data invariant (see~e.g.~ \cite{UH-inverse-problems-seeing-the-unseen}). This gauge invariance is given in terms of some diffeomorphisms that fix the boundary. There is also another type of gauge invariance for magnetic potentials in the magnetic Calderón problem (see~e.g.~ \cite{krupchyk2012uniqueness}). Similar gauge has been observed in the fractional Calderón problem for general elliptic operators \cite{ghosh2021calderon} and on Riemannian manifolds \cite{feizmohammadi2021fractional,feizmohammadiEtAl2021fractional}. The magnetic fractional Calderón problem studied in \cite{CO-magnetic-fractional-schrodinger} enjoys also a gauge analogously to its classical counterparts. In all of these problems, this kind of gauge is, more or less, expected by the physical (or mathematical) structure of the problem and they should not be considered as counterexamples to the uniqueness in the usual sense.

Counterexamples to the anisotropic Calderón problem has been studied recently on Riemannian manifolds in the series of works \cite{DKF-Survey-Non-uniqueness,DKN-Non-Uniqueness-Aniso,DKN-Hidden-Mechanism-Non-Uniqueness,DKN2019nonuniqueness,DKN-anisotropic-Calderon}. There are also recent counterexamples to inverse problems for the wave equation \cite{LO-Counterex-Wave-eq}. Counterexamples to Calderón problems with sufficiently irregular coefficients are known to exist, see~e.g.~ a discussion about quantum shielding \cite[Section 3.1]{UH-inverse-problems-seeing-the-unseen} and the original works \cite{Greenleaf2001TheCP,Greenleaf2003OnNF}. See also \cite{Greenleaf10169,GKLU-Invisibility} and references therein for extensive studies of invisibility in inverse problems, transformation optics and their possible applications, including invisibility cloaking. This article is the first work on counterexamples to the fractional Calderón problems to the best of our knowledge.

\subsection*{Acknowledgements} J.R. was supported by the Vilho, Yrjö and Kalle Väisälä Foundation of the Finnish Academy of Science and Letters.

\section{Preliminaries} \label{sec: preliminaries}

In this section, we introduce the basic notation and known basic properties used in this article. In particular, we introduce the (local) Bessel potential spaces, the fractional Laplacian, the fractional gradient and the fractional divergence, which are needed to set up the inverse problem for the fractional conductivity equation. 

\subsection{Function spaces}

Throughout the article $\Omega, F \subset \R^n$ are open and closed sets, respectively. The space of Schwartz functions is denoted by $\schwartz(\R^n)$ and its dual space, the space of tempered distributions, by $\tempered(\R^n)$. On the space of Schwartz functions the Fourier transformation acts as an isomorphism and is given by
\[
    \fourier u(\xi)\vcentcolon = \int_{\R^n} u(x)e^{-ix \cdot \xi} \,dx
\]
for $u\in\schwartz(\R^n)$, which we frequently also denote by $\hat{u}$. By duality the Fourier transform is also an isomorphism on $\tempered(\R^n)$ and the inverse of the Fourier transform will be denoted by $\ifourier$. 

Next we introduce the Bessel potential spaces $H^{s,p}(\R^n)$ and different local versions of them which can be seen heuristically as having vanishing trace. We define the Bessel potential of order $s \in \R$ as a Fourier multiplier $\vev{D}^s\colon \tempered(\R^n) \to \tempered(\R^n)$ by
\begin{equation}\label{eq: Bessel pot}
    \vev{D}^s u \vcentcolon = \ifourier(\vev{\xi}^s\hat{u}),
\end{equation} 
where $\vev{\xi}\vcentcolon = (1+|\xi|^2)^{1/2}$. For any $1 \leq p < \infty$ and $s \in \R$ we denote the Bessel potential space $H^{s,p}(\R^n)$ by
\begin{equation}
\label{eq: Bessel pot spaces}
    H^{s,p}(\R^n) \vcentcolon = \{\, u \in \tempered(\R^n)\,;\, \vev{D}^su \in L^p(\R^n)\,\}
\end{equation}
which is endowed with the norm $\norm{u}_{H^{s,p}(\R^n)} \vcentcolon = \norm{\vev{D}^su}_{L^p(\R^n)}$. We will use the following variants of \emph{local Bessel potential spaces}:
\begin{equation}\begin{split}\label{eq: local bessel pot spaces}
    \widetilde{H}^{s,p}(\Omega) &\vcentcolon = \overline{C_c^\infty(\Omega)}^{H^{s,p}(\R^n)},\\
    H_F^{s,p}(\R^n) &\vcentcolon =\{\,u \in H^{s,p}(\R^n)\,;\, \supp(u) \subset F\,\}.
    \end{split}
\end{equation}
The spaces $\widetilde{H}^{s,p}(\Omega)$, $H_F^{s,p}(\R^n)$ are closed subspaces of $H^{s,p}(\R^n)$. As usual we set $H^s \vcentcolon = H^{s,2}$. Moreover, for any $s\geq 0$ we set
\[
    \Hcirc(\Omega)=\{\,u \in H^{s}(\R^n)\,;\, u=0\,\,\text{a.e. in}\,\,\Omega^c\,\}
\]
and it is again a closed subspace of $H^s(\R^n)$. 

\subsection{Fractional Laplacian and fractional gradient}

For any tempered distribution $u\in\tempered(\R^n)$, we define the fractional Laplacian of order $s>0$ as the Fourier multiplier
\[
    (-\Delta)^su\vcentcolon = \ifourier(|\xi|^{2s}\hat{u}),
\]
if the right hand side is well-defined. One can show that the fractional Laplacian is a well-defined bounded linear operator $(-\Delta)^{s}\colon H^{t,p}(\R^n) \to H^{t-2s,p}(\R^n)$ for all $1 \leq p < \infty$, $s \geq 0$ and $t \in \R$. Moreover, if $u\in\schwartz(\R^n)$ and $0<s<1$, then the fractional Laplacian can equivalently be calculated as the singular integral (see~e.g.~~\cite[Section 3]{DINEPV-hitchhiker-sobolev})
\[
    (-\Delta)^su(x)=C_{n,s}\,\text{p.v.}\int_{\R^n}\frac{u(x)-u(y)}{|x-y|^{n+2s}}\,dy=C_{n,s}\,\lim_{\epsilon\to 0}\int_{\R^n\setminus B_{\epsilon}(x)}\frac{u(x)-u(y)}{|x-y|^{n+2s}}\,dy
\]
or as an integral of a weighted second order difference quotient
\[
    (-\Delta)^su(x)=-\frac{C_{n,s}}{2}\int_{\R^n}\frac{u(x+y)+u(x-y)-2u(x)}{|y|^{n+2s}}\,dy,
\]
where the constant $C_{n,s}$ is given by
\[
    C_{n,s}=\left(\int_{\R^n}\frac{1-\cos(x_1)}{|x|^{n+2s}}\,dx\right)^{-1}<\infty.
\]
Next we recall the notion of fractional gradient and divergence (see~e.g.~\cite{covi2019inverse-frac-cond, NonlocDiffusion, RZ2022unboundedFracCald}). For any $0<s<1$ the fractional gradient of order $s$ is the bounded linear operator $\nabla^s\colon H^s(\R^n)\to L^2(\R^{2n};\R^n)$ given by (cf.~\cite[Propositions 3.4 and 3.6]{DINEPV-hitchhiker-sobolev})
    \[
        \nabla^su(x,y)=\sqrt{\frac{C_{n,s}}{2}}\frac{u(x)-u(y)}{|x-y|^{n/2+s+1}}(x-y)
    \]
    with
    \begin{equation}
    \label{eq: bound on fractional gradient}
        \|\nabla^su\|_{L^2(\R^{2n})}=\|(-\Delta)^{s/2}u\|_{L^2(\R^n)}\leq \|u\|_{H^s(\R^n)}
    \end{equation}
    where we write in the norms simply $L^2(\R^{2n})$ in the place of $L^2(\R^{2n};\R^n)$.
    The adjoint operator of the fractional gradient is called the fractional divergence and it is the linear bounded operator $\Div_s\colon L^2(\R^{2n};\R^n)\to H^{-s}(\R^n)$ given by
    \[
        \langle \Div_s(u),v\rangle_{H^{-s}(\R^n)\times H^s(\R^n)}=\langle u,\nabla^sv\rangle_{L^2(\R^{2n})}
    \]
    for all $u\in L^2(\R^{2n};\R^n),v\in H^s(\R^n)$. Moreover, one can show that (cf.~\cite[Section 8]{RZ2022unboundedFracCald})
    \[
        \|\Div_s(u)\|_{H^{-s}(\R^n)}\leq \|u\|_{L^2(\R^{2n})}
    \]
    for all $u\in L^2(\R^{2n};\R^n)$ and there holds (cf.~\cite[Lemma 2.1]{covi2019inverse-frac-cond}) $(-\Delta)^su=\Div_s(\nabla^su)$ in the weak sense for all $u\in H^s(\R^n)$.

\section{Invariance of data}\label{sec: invariance of data}

    Next we want to make the point of Theorem~\ref{thm: characterization of uniqueness} \ref{item 2 characterization of uniqueness} on the invariance of data more precise. Fix some $\gamma_1\in L^{\infty}(\R^n)$ with background deviation $m_1\in H^{2s,\frac{n}{2s}}(\R^n)$ satisfying the assumptions of Theorem~\ref{thm: characterization of uniqueness}, let $W_1,W_2\subset \Omega_e$ be two nonempty, disjoint, open sets and introduce the operators
\[
    T_{\gamma_1}\colon X\to H^{2s,n/2s}(\R^n)+H^s(\R^n),\quad m_0\mapsto T_{\gamma_1}m_0\vcentcolon= m_1-m
\]
and 
\[
    S_{\gamma_1}\colon X\to L^1_{loc}(\R^n),\quad S_{\gamma_1}m_0\vcentcolon = T_{\gamma_1}m_0+1,
\]
where $m\in H^s(\R^n)$ is the unique solution to \eqref{eq: PDE uniqueness cond eq main thm}. Next define $Z_F\vcentcolon = T_{\gamma_1}^{-1}(Y_F)$, where $F\subset \R^n$ is a closed set, $Y_F\vcentcolon = H^{2s,\frac{n}{2s}}_F(\R^n)$, and set
\[
    A_{\gamma_1,W_1\cup W_2}\vcentcolon = \bigcup_{\alpha>0}(S_{\gamma_1}(Z_{(W_1\cup W_2)^c})\cap L^{\infty}_{\alpha}(\R^n)),
\]
where $\Gamma\in L^{\infty}_{\alpha}(\R^n)$ if $\Gamma\in L^{\infty}(\R^n)$ and $\Gamma(x)\geq \alpha$ for a.e. $x\in\R^n$. 

Note that $\gamma_1^{1/2}\in A_{\gamma_1,W_1\cup W_2}$. In fact, by the assumptions, we know $\gamma_1^{1/2}\in L^{\infty}_{\gamma_0^{1/2}}(\R^n)$ and $\gamma_1^{1/2}\in S_{\gamma_1}(Z_{(W_1\cup W_2)^c})$ holds since it is equivalent to $m_1\in Y_{(W_1\cup W_2)^c}\cap T_{\gamma_1}(X)$ but this is clear by the construction of $T_{\gamma_1}$. 

Moreover, if $\Gamma_2\in A_{\gamma_1,W_1\cup W_2}$, then $\gamma_2\vcentcolon=\Gamma_2^2$ satisfies the assumptions of Theorem~\ref{thm: characterization of uniqueness} for some $\alpha_0>0$ and is a solution to \eqref{eq: PDE uniqueness cond eq main thm}. Therefore, using \ref{item 2 characterization of uniqueness} of Theorem~\ref{thm: characterization of uniqueness}, we deduce that $\left.\Lambda_{\gamma_1}f\right|_{W_2}=\left.\Lambda_{\gamma_2}f\right|_{W_2}$ for all $f\in C_c^{\infty}(W_1)$. On the other hand, if $\gamma_2$ fulfills the assumptions of Theorem~\ref{thm: characterization of uniqueness} and $\left.\Lambda_{\gamma_1}f\right|_{W_2}=\left.\Lambda_{\gamma_2}f\right|_{W_2}$ for all $f\in C_c^{\infty}(W_1)$, then again by \ref{item 2 characterization of uniqueness} of Theorem~\ref{thm: characterization of uniqueness} it follows that $\Gamma_2\vcentcolon=\gamma_2^{1/2} \in A_{\gamma_1,W_1\cup W_2}$. Hence, if we can show that $A_{\gamma_1,W_1\cup W_2}\setminus\{\gamma_1^{1/2}\}$ is nonemtpy for some $\gamma_1$ satisfying the assumptions of Theorem~\ref{thm: characterization of uniqueness}, then we have constructed a counterexample for uniqueness of the inverse problem related to the fractional conductivity equation.

In particular, $A_{\gamma_1,W_1\cup W_2}$ is the largest set of sufficiently regular conductivities (when squared) so that the characterization of Theorem~\ref{thm: characterization of uniqueness} is known to apply, and the exterior data stays invariant.

\section{Construction of counterexamples}\label{sec: counterexamples}

In this final section of the article, we construct the counterexamples. We make a few remarks about the construction now. First of all, using some general arguments and regularity theory it seems to be extremely difficult, if even possible, to construct counterexamples and verify the required regularity properties even in the case of a bounded domain. For this reason, we had to came up with our construction as it is. The construction shows that there is a lot of freedom \emph{how to start} constructing counterexamples but the construction here gives only smooth counterexamples. However, less regular counterexamples could be also constructed by using rougher radial mollifiers. The construction does not have a good control of $\gamma$ near $\partial \Omega$ or even in $\Omega$. Therefore it cannot be used, at least directly, to show that for \emph{all} nonempty, disjoint, open sets $W_1,W_2 \subset \Omega_e$ one can also construct counterexamples to uniqueness. Moreover, our proof of Theorem~\ref{thm: counterexample} uses the maximum principle on bounded sets and, therefore, the argument may not directly hold on the domains that are bounded in one direction. In the proof of Theorem~\ref{thm: counterexample 2}, we apply a similar strategy to construct counterexamples for domains $\Omega$ which are bounded in one direction as in the proof of Theorem~\ref{thm: counterexample}, but this time using a scaling argument and basic estimates instead of the boundedness of $\Omega$. Lastly, without knowing Theorem \ref{thm: characterization of uniqueness} we expect it would be a hard problem to verify directly that the constructed counterexamples are really counterexamples to uniqueness, i.e. that the partial exterior DN maps coincide.

\begin{proof}[Proof of Theorem \ref{thm: counterexample}]
    Throughout the proof we denote by $A_{\delta}$ the open $\delta-$neighborhood for any set $A\subset\R^n$ (this should not be confused with the notation $\Omega_e$ for the exterior), $\delta>0$ and by $(\rho_{\epsilon})_{\epsilon>0}$ the standard mollifiers. First assume that the conductivities $\gamma_1,\gamma_2\in L^{\infty}(\R^n)$ with background deviations $m_1,m_2\in H^{2s,\frac{n}{2s}}(\R^n)\cap H^s(\R^n)$ satisfy the assumptions of Theorem~\ref{thm: characterization of uniqueness} for some $\gamma_0>0$ and $m\vcentcolon =m_1-m_2\in H^s(\R^n)$ is a solution to \eqref{eq: PDE uniqueness cond eq main thm} with $m_0\in H^s(\R^n)$. Then a direct computation shows that $m_2\in H^s(\R^n)$ solves 
    \[
    (-\Delta)^sm_2-\frac{(-\Delta)^sm_1}{\gamma_1^{1/2}}m_2=\frac{(-\Delta)^sm_1}{\gamma_1^{1/2}}\quad \text{in}\quad \Omega,\quad m_2=m_1-m_0\quad\text{in}\quad \Omega_e.
    \]
    In fact, there holds
    \begin{align*}
        0&=(-\Delta)^sm-\frac{(-\Delta)^sm_1}{\gamma_1^{1/2}}m\\
        &=(-\Delta)^sm_1-(-\Delta)^sm_2-\frac{(-\Delta)^sm_1}{\gamma_1^{1/2}}m_1+\frac{(-\Delta)^sm_1}{\gamma_1^{1/2}}m_2\\
        &=(-\Delta)^sm_1-(-\Delta)^sm_2-\frac{(-\Delta)^sm_1}{\gamma_1^{1/2}}(\gamma_1^{1/2}-1)+\frac{(-\Delta)^sm_1}{\gamma_1^{1/2}}m_2\\
        &=-(-\Delta)^sm_2+\frac{(-\Delta)^sm_1}{\gamma_1^{1/2}}+\frac{(-\Delta)^sm_1}{\gamma_1^{1/2}}m_2
    \end{align*}
    in $\Omega$. Next let $\gamma_1\equiv 1$. Then $m_2\in H^s(\R^n)$ solves
    \begin{equation}
    \label{eq: constr m}
        (-\Delta)^sm_2=0\quad \text{in}\quad \Omega,\quad m_2=m_0\quad\text{in}\quad \Omega_e,
    \end{equation}
    where we replaced $m_0$ by $-m_0$ for later convenience. 
    
    Next choose $\omega\Subset \Omega_e\setminus\overline{W_1\cup W_2}$ and fix $\epsilon>0$ such that $\Omega_{5\epsilon},\omega_{5\epsilon}\subset \R^n\setminus (W_1\cup W_2)$ are disjoint. Now we may take a smooth open set $\Omega'\subset \R^n$ such that $\overline{\Omega}_{2\epsilon}\subset\Omega'\subset \Omega_{3\epsilon}$ and a nonnegative cut-off function $\eta\in C_c^{\infty}(\omega_{3\epsilon})$ with $\eta|_{\overline{\omega}_{2\epsilon}}=1$. These sets are graphically illustrated in Figure~\ref{fig: Geometric setting 2}. By the Lax--Milgram theorem there is a unique solution $\tilde{m}_2\in H^s(\R^n)$ to
    \begin{equation}
    \label{eq: PDE in extended domain}
        (-\Delta)^s\tilde{m}_2=0\quad \text{in}\quad \Omega',\quad \tilde{m}_2=\eta\quad\text{in}\quad \Omega'_e.
    \end{equation}
    Note that $\tilde{m}_2=\eta$ a.e. in $\R^n\setminus\Omega'$ and hence $\supp(\tilde{m}_2)\subset \Omega'\cup \omega_{3\epsilon}$. By \cite[Lemma 3.15]{chandlerwilde2017sobolev} we have $\tilde{H}^s(\Omega')=\Hcirc(\Omega')$ since this is true for all $C^0$ domains, and hence our notion of weak solutions coincide with the one used in \cite{rosoton2015nonlocal}. Therefore, as $\eta\geq 0$, we can apply \cite[Proposition 4.1]{rosoton2015nonlocal} to conclude that $\tilde{m}_2\geq 0$ a.e. in $\R^n$. 
    
    By the mapping properties of the fractional Laplacian and mollification we have
    \[
        (-\Delta)^{s/2}\phi\in L^p(\R^n)\quad \text{and}\quad \supp(\rho_{\epsilon}\ast \phi)\subset \overline{B_{\epsilon}}+\supp(\phi)\subset \Omega_{2\epsilon}\subset \Omega'
    \]
for any $\phi\in C_c^{\infty}(\Omega)$ and any $1\leq p\leq \infty$. Therefore, for any $\phi\in C_c^{\infty}(\Omega)$ we can test \eqref{eq: PDE in extended domain} with $\rho_{\epsilon}\ast \phi\in C_c^{\infty}(\Omega')$ and obtain
\[
\begin{split}
    0&=\int_{\R^n}(-\Delta)^{s/2}\tilde{m}_2(-\Delta)^{s/2}(\rho_{\epsilon}\ast \phi)\,dx=\int_{\R^n}(-\Delta)^{s/2}\tilde{m}_2(\rho_{\epsilon}\ast (-\Delta)^{s/2}\phi)\,dx\\
    &=\int_{\R^n}(\rho_{\epsilon}\ast(-\Delta)^{s/2}\tilde{m}_2) (-\Delta)^{s/2}\phi\,dx=\int_{\R^n}(-\Delta)^{s/2}(\rho_{\epsilon}\ast\tilde{m}_2) (-\Delta)^{s/2}\phi\,dx,
\end{split}
\]
where we have used that the mollifier $\rho_{\epsilon}$ is radial and by Young's inequality the fractional Laplacian satisfies
\begin{equation}
\label{eq: fractional Lap commutes with conv}
    (-\Delta)^{s/2}(\rho_{\epsilon}\ast u)=\rho_{\epsilon}\ast (-\Delta)^{s/2}u,\, (-\Delta)^{s/2}(\rho_{\epsilon}\ast v)=((-\Delta)^{s/2}\rho_{\epsilon})\ast v\in L^q(\R^n)
\end{equation}
for all $u\in H^{s,p}(\R^n),v\in L^p(\R^n)$ and $1\leq p,q\leq \infty$ with $q\geq p$. Hence, we have shown that $m_2\vcentcolon = \tilde{m}_2\ast \rho_{\epsilon}\in H^s(\R^n)$ solves
\[
(-\Delta)^sm_2=0\quad \text{in}\quad \Omega,\quad m_2=\tilde{m}_2\ast \rho_{\epsilon}\quad\text{in}\quad \Omega_e.
\]
Furthermore, note that $m_2$ has the following properties:
\begin{enumerate}[(i)]
    \item\label{item 1 m} $m_2\geq 0$ a.e. in $\R^n$,
    \item\label{item 2 m} $m_2\in L^{\infty}(\R^n)$,
    \item\label{item 3 m} $m_2\in H^{s}(\R^n)\cap H^{2s,\frac{n}{2s}}(\R^n)$
    \item\label{item 4 m} and $\supp(m_2)\subset \R^n\setminus(W_1\cup W_2)$.
\end{enumerate}
The first assertion \ref{item 1 m} follows from $\tilde{m}_2\geq 0$ and $\rho_{\epsilon}\geq 0$. Next observe that $\supp(\tilde{m}_2)\subset \Omega'\cup \omega_{3\epsilon}$ implies $\supp(m_2)\subset \supp(\rho_{\epsilon}\ast \tilde{m}_2)\subset \overline{B_{\epsilon}}+\supp(\tilde{m}_2)\subset \Omega_{5\epsilon}\cup \omega_{5\epsilon}\subset\R^n\setminus(W_1\cup W_2)$ and therefore $m_2\in C_c^{\infty}(\R^n)$. Hence, we have shown that the statements \ref{item 2 m}--\ref{item 4 m} hold. Moreover, the support conditions imply that $\gamma_2=1$ in $W_1\cup W_2$.

This implies that the conductivity $\gamma_2$ defined by $\gamma_2^{1/2}\vcentcolon = m_2+1$ and the background deviation $m_2$ satisfy all required properties but $\gamma_1\neq \gamma_2$. Now since $W_1,W_2\subset \Omega_e$ are two disjoint open sets, we have found two conductivities $\gamma_1,\gamma_2$ satisfying the properties of Theorem~\ref{thm: characterization of uniqueness} and $m\vcentcolon = m_1-m_2$ solving \eqref{eq: PDE uniqueness cond eq main thm}, which in turn implies that the induced DN maps satisfy $\Lambda_{\gamma_1}f|_{W_2}=\Lambda_{\gamma_2}f|_{W_2}$ for all $f\in C_c^{\infty}(W_1)$.
\end{proof}

 \begin{figure}[!ht]
    \centering
    \begin{tikzpicture}[scale=1.4]
    \draw[color=cyan!50, fill=cyan!10] plot [smooth cycle, tension=0.05] coordinates {(0,0) (0,3) (2,3) (2,2)  (1,2) (1,1) (2,1) (2,0)};
    \filldraw[color=red!60, fill=red!20, opacity = 0.4] plot [smooth cycle, tension=0.2] coordinates {(-0.12,-0.1) (-0.12,3.1) (2.12,3.12) (2.11,1.85) (1.15,1.85) (1.15,1.15) (2.12,1.15) (2.12,-0.15) (1,-0.19)};
    \draw[color=cyan!50] plot [smooth cycle, tension=0.05] coordinates { (-0.1,-0.1) (-0.1,3.1) (2.1,3.1) (2.1,1.9) (1.1,1.9) (1.1,1.1) (2.1,1.1) (2.1,-0.1)};
    \draw[color=cyan!50] plot [smooth cycle, tension=0.05] coordinates {  (-0.2,-0.2)  (-0.2,3.2)  (2.2,3.2) (2.2,1.8)  (1.2,1.8) (1.2,1.2) (2.2,1.2) (2.2,-0.2)};
    \draw[color=cyan!50] plot [smooth cycle, tension=0.05] coordinates {   (-0.3,-0.3) (-0.3,3.3) (2.3,3.3) (2.3,1.7) (1.3,1.7) (1.3,1.3)  (2.3,1.3) (2.3,-0.3)};
    \draw[color=cyan!50] plot [smooth cycle, tension=0.05] coordinates {   (-0.4,-0.4) (-0.4,3.4) (2.4,3.4) (2.4,1.6) (1.4,1.6) (1.4,1.4) (2.4,1.4) (2.4,-0.4)};
    \draw[color=gray!50] (-0.03,2.1)--(-0.7,2.1);
    \node at (-0.95,2.1) {$\raisebox{-.35\baselineskip}{\ensuremath{\Omega}}$};
    \draw[color=gray!50] (-0.2,1.6)--(-0.7,1.6);
    \node at (-0.9,1.6) {$\raisebox{-.35\baselineskip}{\ensuremath{\Omega'}}$};
    \filldraw[color=orange!50, fill=orange!10](5.65,2.25) circle (0.7);
    \filldraw[color=blue!50, fill=blue!5, opacity=0.5] (7,1) ellipse (0.8 and 0.6);
    \node at (5.65,2.25) {$\raisebox{-.35\baselineskip}{\ensuremath{W_2}}$};
    \node at (7,1) {$\raisebox{-.35\baselineskip}{\ensuremath{W_1}}$};
    \filldraw[color=green!50, fill=green!20] (4,1) ellipse (0.3 and 0.7);
    \draw[color=green!50] (4,1) ellipse (0.4 and 0.8);
    \draw[color=green!50] (4,1) ellipse (0.5 and 0.9);
    \draw[color=green!50] (4,1) ellipse (0.6 and 1);
    \draw[color=green!50] (4,1) ellipse (0.7 and 1.1);
    \draw[color=green!50] (4,1) ellipse (0.8 and 1.2);
     \node at (4,1) {$\raisebox{-.35\baselineskip}{\ensuremath{\omega}}$};
\end{tikzpicture}\label{fig: Geometric setting 2}
    \caption{A graphical illustration of the sets used in the proof of Theorem~\ref{thm: counterexample}. Here $\Omega\subset \R^n$ is an arbitrary open bounded set and the measurements are performed in the disjoint open sets $W_1,W_2\subset\Omega_e$. In the proof, we construct in the first step a nonzero $s$-harmonic background deviation $\tilde{m}_2\in H^s(\R^n)$ in the set $\Omega'$, which has a smooth boundary and lies in the deformed annulus $\Omega_{3\epsilon}\setminus\overline{\Omega}_{2\epsilon}$, and then obtain by mollification a nonzero smooth $s-$harmonic function $m_2\vcentcolon =\tilde{m}_2\ast \rho_{\epsilon}$ in the set $\Omega$. The set $\omega\Subset\Omega_e\setminus\overline{W_1\cup W_2}$ is used to construct a cut--off function $\eta\in C_c^{\infty}(\omega_{3\epsilon})$ with $\eta|_{\overline{\omega}}=1$, which $\tilde{m}_2$ has as an exterior value, and in the end this function leads to the property that $m_2$ is nonnegative and has compact support contained in $\Omega_{5\epsilon}\cup \omega_{5\epsilon}$. Note that the topology of the sets $\Omega$ and $\omega$ could be also very complicated.}
\end{figure}
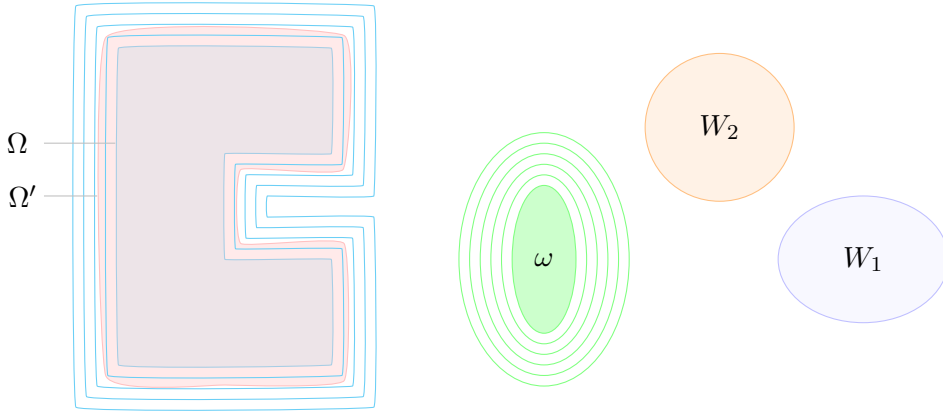

\begin{proof}[{Proof of Theorem~\ref{thm: counterexample 2}}]
    The construction in this case is very similar to the one of Theorem~\ref{thm: counterexample}. Without loss of generality, we can assume that $n\geq 2$, since the result for $n=1$ is already established in Theorem~\ref{thm: counterexample}. Let the sets $\Omega_{k\epsilon}, \omega,\omega_{k\epsilon}$ for $k\in\{1,\ldots,5\}$, $\epsilon>0$ and the cut--off function $\eta$ be defined as in the proof of Theorem~\ref{thm: counterexample}. 
    
    Let $\gamma_1\equiv 1$ and suppose $\tilde{m}_2\in H^s(\R^n)$ is the unique solution to  
    \begin{equation}
    \label{eq: PDE in unbounded domain}
        (-\Delta)^s\tilde{m}_2=0\quad \text{in}\quad \Omega_{2\epsilon},\quad \tilde{m}_2=\eta\quad\text{in}\quad \R^n\setminus\Omega_{2\epsilon}.
    \end{equation}
    As in the proof of Theorem~\ref{thm: counterexample}, the function $\tilde{m}_2$ uniquely exists by the Lax--Milgram theorem combined with the fractional Poincar\'e inquality on domains bounded in one direction (cf.~\cite[Theorem 2.2]{RZ2022unboundedFracCald}). By Young's inequality we have $\rho_{\epsilon}\ast \tilde{m}_2\in L^{\infty}(\R^n)$ with 
    \[
        \|\rho_{\epsilon}\ast \tilde{m}_2\|_{L^{\infty}(\R^n)}\leq \|\rho_{\epsilon}\|_{L^2(\R^n)}\|\tilde{m}_2\|_{L^2(\R^n)}
    \]
    and hence one easily verifies that the function 
    \[
        m_2\vcentcolon =C_{\epsilon}\rho_{\epsilon}\ast\tilde{m}_2\in H^s(\R^n)\quad\text{with}\quad C_{\epsilon}\vcentcolon=\frac{\epsilon^{n/2}}{2|B_1|^{1/2}\|\rho\|_{L^{\infty}(\R^n)}^{1/2}\|\tilde{m}_2\|_{L^2(\R^n)}}
    \]
    satisfies $\|m_2\|_{L^{\infty}(\R^n)}\leq 1/2$. Proceeding as in the proof of Theorem~\ref{thm: counterexample}, we deduce that $m_2\in H^s(\R^n)$ weakly solves
    \[
        (-\Delta)^sm=0\quad \text{in}\quad \Omega,\quad m=m_2\quad\text{in}\quad \Omega_e.
    \]
    
    If $n\geq 4$ or $n\geq 4s$ for $n=2,3$ we have $p\vcentcolon =\frac{2n}{n+4s}\geq 1$ and there holds $1/p+1/2=1+2s/n$. Therefore, by Young's inequality and the fact that the Bessel potential commutes with convolution we obtain $m_2\in H^{2s,\frac{n}{2s}}(\R^n)$. Moreover, by definition we have $\supp(\tilde{m}_2)\subset \overline{\Omega}_{2\epsilon}\cup\omega_{3\epsilon}$ and therefore there holds $\supp (m_2)\subset \supp(\tilde{m}_2)+\overline{B}_{2\epsilon}\subset\overline{\Omega}_{4\epsilon}\cup\omega_{5\epsilon}\subset\Omega_{5\epsilon}\cup\omega_{5\epsilon}$. In particular, $m_2$ vanishes in $W_1\cup W_2$. Since $\|m_2\|_{L^{\infty}(\R^n)}\leq 1/2$, the function defined by $\gamma_2^{1/2}\vcentcolon = 1+m_2\in L^{\infty}(\R^n)$ satisfies $\gamma_2^{1/2}\geq 1/2$. Therefore, the conductivities $\gamma_1,\gamma_2$ with background deviations $m_1,m_2$ satisfy all required properties and we can conclude the proof.
\end{proof}

\bibliography{countrefs} 

\bibliographystyle{alpha}

\end{document}